\documentclass[runningheads]{llncs}
\usepackage{latexsym}
\usepackage{amsmath}
\usepackage{algorithmicx}
\usepackage{amsrefs}
\usepackage{amscd}
\usepackage{amssymb}
\usepackage{mathrsfs}
\usepackage{mathtools}
\usepackage{color,framed}
\usepackage[ruled]{algorithm2e} 
\usepackage{graphicx}
\usepackage{subfigure}
\usepackage{url}
\newcommand{\R}{\mathbb{R}}

\newcommand{\M}{\mathcal{M}}
\newcommand{\Q}{\mathcal{Q}} 
\newcommand{\G}{\mathcal G}

\newcommand{\g}{\mathfrak g}

\newcommand{\Tr}{\mathrm{Tr}}
\newcommand{\Id}{\mathrm{Id}}
\newcommand{\hvf}[1]{\setbox0=\hbox{$#1$}%
  \ifdim\wd0>1em\widehat{#1}\else\hat{#1}\fi} 

\newcommand{\gen}{\mathcal{L}} 
  
\def\rational#1#2{{\mathchoice{\textstyle{#1\over#2}}%
  {\scriptstyle{#1\over#2}}{\scriptscriptstyle{#1\over#2}}{#1/#2}}}
\def\half{\rational12}		

\def\diff{\mathrm{d}}

\def\defn{:=}

\def\tang{\partial}

\begin{document}
\title{Irreversible Langevin MCMC on Lie Groups}
\author{Alexis Arnaudon\inst{1} \and
Alessandro Barp\inst{1,2} \and
So Takao\inst{1}}
\authorrunning{Arnaudon, Barp and Takao}
\institute{Imperial College London \and
Alan Turing Institute
}
\maketitle              
\begin{abstract}
It is well-known that irreversible MCMC algorithms converge faster to their stationary distributions than reversible ones. 
Using the special geometric structure of Lie groups $\G$ and dissipation fields compatible with the symplectic structure, we construct an irreversible HMC-like MCMC algorithm on $\G$,
 where we first update the momentum by solving an OU process on the corresponding Lie algebra $\mathfrak g$, and then approximate the Hamiltonian system on $\G\times \g$ with a reversible symplectic integrator followed by a Metropolis-Hastings correction step.
In particular, when the OU process is simulated over sufficiently long times, we recover HMC as a special case. 
 We illustrate this algorithm numerically using the example $\G = SO(3)$.

\end{abstract}

\section{Introduction}

In this work, we construct an irreversible MCMC algorithm on Lie groups, which generalises the standard Hamiltonian Monte Carlo (HMC) algorithm on $\mathbb R^n$.
The HMC  method~\cite{duane1987hybrid} generates samples from a probability density (with respect to an appropriate reference measure) known up to a constant factor by generating proposals using Hamiltonian mechanics,
 which is approximated by a reversible symplectic numerical integrator and followed by a Metropolis-Hastings step to correct for the bias introduced during the numerical approximations.
The resulting time-homogeneous Markov chain is thus reversible, 
and allow distant proposals to be accepted with high probability,
 which decreases the correlations between samples (for a basic reference on HMC see \cite{neal2011mcmc}, and for a geometric description see \cite{barp2018geometry,betancourt2017geometric}).
However, it is well-known that ergodic irreversible diffusions converge faster to their target distributions~\cite{duncan2016variance,rey2015irreversible}, 
and several irreversible MCMC algorithms based on Langevin dynamics have been proposed~\cite{ottobre2016markov,ottobre2016function}.

From a mechanical point of view, diffusions on Lie groups are important since they form the configuration space of many interesting systems, such as the free rigid body. 
For example in \cite{chen2015constrained} Euler-Poincar\'e reduction of group invariant symplectic diffusions on Lie groups are considered in view of deriving dissipative equations from a variational principle, and in \cite{arnaudon2018noise} Langevin systems on coadjoint orbits are constructed by adding noise and dissipation to Hamiltonian systems on Lie groups. 
The phase transitions of this system were analysed using a  sampling method~\cite{arnaudon2018networks}. 
In lattice gauge theory one typically uses the HMC algorithm for semi-simple compact Lie groups which was originally presented in \cite{kennedy88b} and extended to arbitrary Lie groups in~\cite{barp2019hamiltonianB}, see also~\cite{Kennedy:2012,clark2008tuning,clark2007accelerating}.
In \cite{barp2019homogeneous}, it was shown how to construct HMC on homogeneous manifolds using symplectic reduction, which includes sampling on Lie groups as a special case. 

To construct an irreversible algorithm on Lie groups, we first extend Langevin dynamics to general symplectic manifolds $\M$ based on Bismut's symplectic diffusion process \cite{bismut1981mecanique}. 
Our generalised Langevin dynamics with multiplicative noise and nonlinear dissipation has the Gibbs measure as the invariant measure,
 which allows us to design MCMC algorithms that sample from a Lie group $\G$ when we take $\M = T^*\G$. 
In our Langevin system the irreversible component is determined by Hamiltonian vector fields which are compatible with the symplectic structure, thus avoiding the appearance of divergence terms associated to the volume distortion.
We are then free to choose the noise-generating Hamiltonians to best suit the target distribution.  
Choosing Hamiltonians that only depend on position allows us to proceed with a 
Strang splitting of the dynamics into a position-dependent OU process in the fibres which can be solved exactly, and a Hamiltonian part which is approximated using a leapfrog scheme, followed by a Metropolis-Hastings acceptance/rejection step in a similar fashion to~\cite{ottobre2016markov,ottobre2016function,bou2010patch}.
 Ideally one wants to choose these Hamiltonians to achieve the fastest convergence to stationarity.

On a general manifold, it would be necessary to introduce local coordinates in order to solve the OU process on the fibres, making it difficult to implement. However, since our base manifold is a Lie group, the Maurer-Cartan form defines an isomorphism between the cotangent bundle $T^*\G$ and the trivial bundle $ \G \times \mathfrak g^*$, which, given an inner-product on $\g$, may further be identified with $\G\times \g$.
As a result, one may pull back the OU process on $T_g^* \G$ to an OU process on $\mathfrak g$
 for any $g \in \G$, thus avoiding the problem of having to choose appropriate charts. Hence on Lie groups, we obtain a practical irreversible MCMC algorithm which generalises the $\mathbb R^n$-version of the irreversible algorithm considered in~\cite{ottobre2016markov,ottobre2016function}.

Finally, we simulate this algorithm in the special case $\G = SO(3)$ and perform a Maximum Mean Discrepancy (MMD) test to show that on average, the irreversible algorithm converges faster to the stationary measure than the corresponding reversible HMC on $SO(3)$.

\section{Diffusions on Symplectic Manifolds}
We consider diffusion processes on symplectic manifolds $(\M,\omega)$, where we have a natural volume form $\omega^n$, and define the canonical Poisson bracket $ \{g,f\}\defn \omega(X_f,X_g)=X_g f $, where $X_g$ is the Hamiltonian vector field associated to the Hamiltonian $g:\M\to \R$, i.e. $\diff g = \iota_{X_g} \omega$.
Given arbitrary functions $H$ and $H_i$ for $i=1, \ldots, m$ on $\M$, we consider the SDE
\begin{equation} \label{symplectic-SDE}
\diff Z_t = \Big( X_{H}(Z_t)-\frac{\beta}{2}\sum_{i=1}^m\{H_i,H\}X_{H_i}(Z_t) \Big) \diff t + \sum_{i=1}^m X_{H_i}(Z_t) \circ \diff W^i_t\, ,
\end{equation} 
which has the generator, or forward Kolmogorov operator 
\begin{align}
&\gen f = -\{ f,H \}  -\frac{\beta}{2} \sum_{i=1}^m \{H,H_i \} \{f,H_i\}+\frac12 \sum_{i=1}^m \{\{f,H_i\},H_i\}\, .
\end{align} 
To show that the Gibbs measure is an invariant measure for \eqref{symplectic-SDE}, we will need the following lemma:
 \begin{lemma}\label{vanishing-bracket}
For a symplectic manifold $(\M,\omega)$ and two functions $f,g \in C^\infty(\M)$ such that either $\partial \M = \varnothing$ or $g|_{\partial M}=0$, we have the following identity
\begin{align*}
\int_{\M} \{f,g\} \omega^n = 0\, .
\end{align*}
\end{lemma}
For the proof, see~\cite{cruzeiro2018momentum}, Section 4.3.
Hence \eqref{symplectic-SDE} enables us to build MCMC algorithms on any symplectic manifold, and in particular
 the cotangent bundle of Lie groups, that converge to the Gibbs measure:
\begin{theorem}
Given a symplectic manifold $(\M,\omega)$ without boundary, equation \eqref{symplectic-SDE} on $\M$ has the Gibbs measure
\begin{align*}
\mathbb P_\infty(z) = p_{\infty} \omega^n \defn \frac1Z e^{-\beta H (z)} \omega^n, \quad Z = \int_\M e^{-\beta H (z)} \omega^n
\end{align*}
as its stationary measure for any choice of $H_i:\M \to \R$ where $i=1,\ldots,m$.
\end{theorem}
\begin{proof}
Using the Leibniz rule $\{fg,h\} = f\{g,h\} + g\{f,h\}$, we have 
\begin{align*}
g\{f,H_i\}\{H ,H_i\}&=\{gf,H_i\}\{H ,H_i\}-f\{g,H_i\}\{H ,H_i\} \\
&= \cdots =
\{gf\{H ,H_i\},H_i\}-f\{g\{H ,H_i\},H_i\}
\end{align*}
and similarly
\begin{align*}
g\{\{f,H_i\},H_i\}=\{g\{f,H_i\},H_i\}-\{f\{g,H_i\},H_i\}+f\{\{g,H_i\},H_i\}.
\end{align*}

Hence one can compute the $L^2(\M,\omega^n)$-adjoint of the operator $\mathcal L$ as follows
\begin{align*}
&\int_{\M} g (\mathcal L f) \,\omega^n = \int_M g\left( -\{f, H \} - \frac{\beta}{2}\sum_{i=1}^m \{f,H_i\}\{H ,H_i\} + \frac12 \sum_{i=1}^m \{\{f,H_i\}, H_i\}\right) \,\omega^n\\
&= \int_{\M} \left(-\{fg,H \} - \frac{\beta}{2}\{fg\{H ,H_i\}, H_i\} + \frac12 \left(\{g\{f,H_i\},H_i\} - \{f\{g,H_i\},H_i\}\right)\right) \omega^n \\
&\quad + \int_{M} f\left(\{g,H \} + \frac{\beta}{2}\{g\{H ,H_i\},H_i\} + \frac12 \{\{g,H_i\},H_i\}\right) \omega^n \\
&= \int_{\M} f\left(\{g,H \} + \frac{\beta}{2}\{g\{H ,H_i\},H_i\} + \frac12 \{\{g,H_i\},H_i\}\right) \omega^n = \int_{\M} f (\mathcal L^* g) \,\omega^n,
\end{align*}
where we have used Lemma \ref{vanishing-bracket} to integrate the Poisson brackets to $0$.
 Hence, we obtain the Fokker-Planck operator
\begin{align*}
\mathcal L^* g = \{g,H \} + \frac{\beta}{2} \sum_{i=1}^m\{g\{H ,H_i\},H_i\} + \frac12 \sum_{i=1}^m\{\{g,H_i\},H_i\}.
\end{align*}

Now, by the derivation property of the Poisson bracket, $\{f \circ g, h\} = f'\circ g \{g,h\}$ and noting that $p_\infty'(H)= -Z^{-1} \beta e^{-\beta H} = -\beta p_\infty(H)$, one can check that
\begin{align*}
\mathcal L^* p_\infty = p_\infty'(H) \{H,H \} + &\frac{\beta}{2} \sum_{i=1}^m\{p_\infty(H)\{H ,H_i\},H_i\} \\
&- \frac{\beta}{2} \sum_{i=1}^m\{p_\infty(H)\{H,H_i\},H_i\} = 0.
\end{align*}
Therefore $\mathbb P_\infty(z) = p_\infty(z) \omega^n$ is indeed an invariant measure for \eqref{symplectic-SDE}.
\end{proof}

If $\M=T^*\Q$ is the cotangent bundle of a manifold without boundary $\Q$, we define the marginal measure $\mathbb P^1_\infty$ on $\Q$ by
\begin{align}
\int_A \mathbb P^1_\infty = \int_{T^*A} \iota^* \mathbb P_\infty,
\end{align}
for any measurable set $A \subset \Q$, where $\iota : T^*A \rightarrow T^*\Q$ is the inclusion map.
In addition, if $(\Q,\gamma)$ is a Riemannian manifold, we can consider the Hamiltonian function $H(q,p) = \frac12 \gamma_q(p,p) + V(q)$, for $(q,p) \in T^*\Q $, and the marginal invariant measure $\mathbb P_\infty^1(\diff q)$ of the process \eqref{symplectic-SDE} is simply
\begin{align*}
\mathbb P_\infty^1(\diff q) = \frac1{Z_1} e^{-V(q)} \sqrt{|g|} \diff q, \quad Z_1 = \int_\Q e^{-V(q)} \sqrt{|g|} \diff q,
\end{align*}
where $\sqrt{|g|} \diff q$ is the Riemannian volume form.

The MCMC algorithm which we will derive in section \S\ref{langevin_G} is based on a Strang splitting of the dynamics \eqref{symplectic-SDE} into a Hamiltonian part and a Langevin part. 
Hereafter, we identify $T^*\Q$ with $T\Q$ through the metric and just consider the dynamics on $T\Q$ instead of $T^*\Q$.


\section{Irreversible Langevin MCMC on Lie Groups}\label{langevin_G}

Consider a $n$ dimensional Lie group $\G$ and let $e_i, \theta^i, i=1,\ldots, n$ be an orthonormal basis of left-invariant vector fields and dual one-forms respectively. 
 We consider $H =V\circ \pi+T : T\G \rightarrow \mathbb R$, 
 where $T$ is the kinetic energy associated to a bi-invariant metric on $\G$ and $V \propto \log \chi : \G \rightarrow \mathbb R$ is the potential energy, where $\chi$ is the distribution we want to sample from on $\G$. We let $v^i:T\G \to \R$ be the fibre coordinate functions with respect to the left-invariant vector fields, $v^i(g,u_g) \defn \theta^i_g(u_g)$.

Vector fields tangent to $T\G$ (i.e., elements of $\Gamma(TT\G)$) can be expanded in terms of left-invariant vector fields $e_i$ and the fibre-coordinate vector fields $\partial_{v^i}$, (i.e., $\Gamma(TT\G) \cong \Gamma(T\G \oplus T\g)$).
We consider noise Hamiltonians that depend only on position, $H_i=U_i \circ \pi$ where $U_i: \G \to \R$, so the corresponding Hamiltonian vector fields can be written as $X_{H_i}=-e_j(U_i) \partial_{v^j}$, (see \cite{barp2019hamiltonianB}).
Hence the stochastic process \eqref{symplectic-SDE} on $T\G$ can be split up into a Langevin part
\begin{align} 
	\diff Z_t &=\frac{\beta}{2} \sum_{i=1}^m X_{H} (H_i)X_{H_i}(Z_t)  \diff t + \sum_{i=1}^m X_{H_i}(Z_t) \circ \diff W^i_t \, , \nonumber \\
&=-\frac{\beta}{2} \sum_{i=1}^m v^k e_k(U_i) e_j(U_i)\partial_{v^j}(Z_t)  \diff t - \sum_{i=1}^me_j(U_i)\partial_{v^j}(Z_t) \circ \diff W^i_t\, .\label{ou-abstract}
\end{align}
and a Hamiltonian part
\begin{align} \label{ham-eq}
\frac{\diff Z_t}{\diff t} = X_H(Z_t).
\end{align}

Note the geodesics are given by the one-parameter subgroups, with Hamiltonian vector field $X_T = v^ke_k$. 
Since $X_{H_i}$ only has components in the fibre direction $\partial_{v^i}$ (i.e., it has no $e_i$ components along $\G$) the diffusion starting at any point $(g,v)\in T\G$ remains in $T_g\G$, i.e., with the same base point $g$. 
When $\G = \R^n$ and we use the standard kinetic energy $T = \frac{1}{2} \|v\|_{\mathbb R^n}^2$, then vector fields become gradients, i.e. $e_j =\partial_{q^j}$, and equation \eqref{ou-abstract} becomes the space dependent Langevin equation for $(q,v) \in T\mathbb R^n$,
\begin{equation}
\dot q=0,\qquad \diff v_t = -\frac{\beta}{2}\nabla_qU_i(q) \nabla_q U_i(q)^T v_t  \diff t-\nabla_qU_i(q) \circ \diff W^i_t\, . 
\end{equation} 

Now let $\xi_i \defn e_i(1)$ be a basis of the Lie algebra $\mathfrak g$, where $1$ is the identity. 
Then $e_i(g) = \tang_1 L_g \xi$
and we can identify $T\G$ with $\G \times \g$ through the relation $(g,v^ie_i(g))\sim (g,v^i \xi_i)$.
In other words, we may now think of $v^i$ as the Lie algebra coordinate functions $v^i:\G\times \g \to \R$ with $v^i(g,u) = \theta^i_1(u)$, and since $\g$ is a vector space, we can identify $\partial_{v^i} \sim \xi_i$ and write 
\begin{align}
    X_{H_i}(g,v)= -e_j|_g(U_i) \xi_j =: \sigma_{ji}(g) \xi_j\, .
\end{align}
for $i = 1,\ldots,m$ and $j = 1, \ldots, n$. For matrix Lie groups, this becomes
\begin{align}
    X_{H_i}(g,v)= - \text{Tr} \big( \nabla U_i^T g \xi_j \big) \xi_j = \sigma_{ji}(g) \xi_j ,
\end{align}
where $(\nabla U_i)_{ab} \defn \partial_{x_{ab}}U_i$, where $x_{ab}$ are the matrix coordinates of $g\in \G$.
The Langevin equation \eqref{ou-abstract} can then be written as 
\begin{equation}
\dot g=0,\qquad \diff v_t = -\frac{\beta}{2} (\sigma(g) \sigma(g) ^T)_{jk} v^k_t \xi_j \diff t + \sigma_{ji}(g)\xi_j \diff W^i_t ,
	\label{OU-G}
\end{equation}
and if we identify $\g \sim \R^n$, $v^i\xi_i \sim \mathbf v \in \mathbb R^n$, we get a standard OU process on $\R^n$
\begin{equation}
 \diff \mathbf v_t = -\frac{\beta}{2} \sigma \sigma ^T \mathbf v_t  \diff t +\sigma \diff \mathbf W_t \, , 
\end{equation}
for the vector-valued Wiener process $\mathbf W_t = (W_t^1, \ldots, W_t^m) \in \mathbb R^m$. Note that the noise term can be interpreted as an It\^o integral since the diffusion coefficient $\sigma$ does not depend on $\mathbf v$.
We can solve this Langevin equation explicitly if the matrix $D = \sigma \sigma^T$ is invertible. 
This is the case if the vectors $\nabla U_i$ form an orthonormal basis of $\mathbb R^n$, or more generally if they satisfy the H\"ormander condition.  
The explicit solution is then given by 
\begin{align}
  v_{t+h} = e^{- \frac{\beta}{2} Dh } v_t + \sigma \int_t^{t+h} e^{-\frac{\beta}{2} D(t+h-s)} \diff \mathbf W_s\, , 
\end{align}
which has the transition probability
\begin{align}
  \begin{split}
  p(v_0,v) &= \frac{1}{(2\pi)^{n/2} |\mathrm{det}{\Sigma_h}|} \,\exp\left({-\frac12 \left \|v - e^{-\frac{\beta}{2} D h }v_0 \right \|_{\Sigma_h^{-1}}^2}\right),\\
  \mathrm{where}\quad   \Sigma_h &= 
  \frac{1}{\beta} \left ( \Id - e^{-\beta D h}\right) \,.
  \end{split}
\end{align}
 
\subsection{MCMC algorithm on Lie groups}

From the Langevin system considered in the previous section, we can construct the following MCMC algorithm to sample from the distribution $\chi := e^{-\frac{\beta}{2}V}$ on $\G$.  
Starting from $(g_0,v_0) \in T\G$, we iterate the following
\begin{enumerate}
	\item Solve equation \eqref{OU-G} exactly until time $h$ by sampling
	  \begin{align}
	    v^* \simeq \mathcal N \left(e^{-\frac{\beta}{2} D h} v_0, \Sigma_h\right)
	  \end{align}
	  to obtain $(\bar{g}_0,\bar{v}_0) = (g_0, v^*)$;
	\item Approximate the Hamiltonian system \eqref{ham-eq} using $N$ Leapfrog trajectories with step size $\delta>0$. 
	Starting at $(\bar{g}_0,\bar{v}_0) = (g_0,v^*)$:\\\\\
{\bf For} $k = 0,\ldots, N-1$: \footnote{For a non-matrix group, simply replace $\Tr\left(\partial_x V^T \bar{g}_k \xi_i\right)$ with $e_i|_g(V)$.}
	      \begin{align*}
  		\bar{v}_{k+\half}&= \bar{v}_k -\frac{\delta}{2}\Tr\left(\partial_x V^T \bar{g}_k \xi_i\right) \xi_i \\
  		\bar{g}_{k+1} &= \bar{g}_k \exp\left(\delta \,\bar{v}_{k+\half} \right) \\
  		\bar{v}_{k+1} &= \bar{v}_{k+\half} -\frac{\delta}{2} \Tr \left(\partial_x V^T \bar{g}_{k+1} \xi_i\right) \xi_i
	      \end{align*}
		to obtain $(\bar{g}_N,\bar{v}_N)$. The time step $\delta$ and number of steps $N$ are to be tuned appropriately by the users.
	\item Accept or reject the proposal by a Metropolis-Hastings step. We accept the proposal $(\bar{g}_N,\bar{v}_N)$ with probability
\begin{align*}
\alpha = \min \left\{1,\exp \left(-H(\bar{g}_N,\bar{v}_N) + H(\bar{g}_0,\bar{v}_0)\right)\right\},
\end{align*}
and set $(g_1,v_1) = (\bar{g}_N,\bar{v}_N)$. On the other hand, if the proposal is rejected, we set $(g_1,v_1) = (\bar{g}_0, -\bar{v}_0)$.
\end{enumerate}

Notice also that in the limit $h\to \infty$, this algorithm become the standard HMC algorithm, which is reversible, but for finite $h$, the algorithm is irreversible (see \cite{ottobre2016markov} for a detailed discussion).

\section{Example on  $SO(3)$}\label{MCMC-SO3}

As an example we will pick the rotation group $\G = SO(3)$, where the Lie algebra $\mathfrak {so}(3)$ consists of $3\times 3$ anti-symmetric matrices so that the kinetic energy on the Lie algebra is $T(v) = \frac{1}{2}\Tr(v^Tv)$, for $v \in \mathfrak{so}(3)$.  
The potentials on $SO(3)$ will be defined by functions $V, U_i:\R^{3\times 3}\to \R$ on matrices. 
We then pick the potential to be of the form $V(g) = e^{\alpha Tr(g)}$ and isotropic noise with 
\begin{align}
	U_i(g) = \epsilon \Tr(e^{-\xi_i} g)\, \qquad\mathrm{for}\qquad  \mathrm{span}(\xi_i) = \mathfrak{so}(3)\, .  
\end{align}
We then obtain samples $g_t$ on $SO(3)$, which we can project onto the sphere by simply letting the group act on a vector $z = (0,0,1)$, to get $x_t = g_tz$, see panel (a) of figure \ref{fig:simulation}.
From these samples $g_t$, we can also estimate the convergence rate of the MCMC algorithm by computing the maximum mean discrepancy (MMD) between the set of first $N$ samples and the whole sequence, using the values on the diagonal of the matrices $g_t$. 
\begin{figure}[htpb]
	\centering
	\subfigure[MCMC samples]{\includegraphics[width=0.3\textwidth]{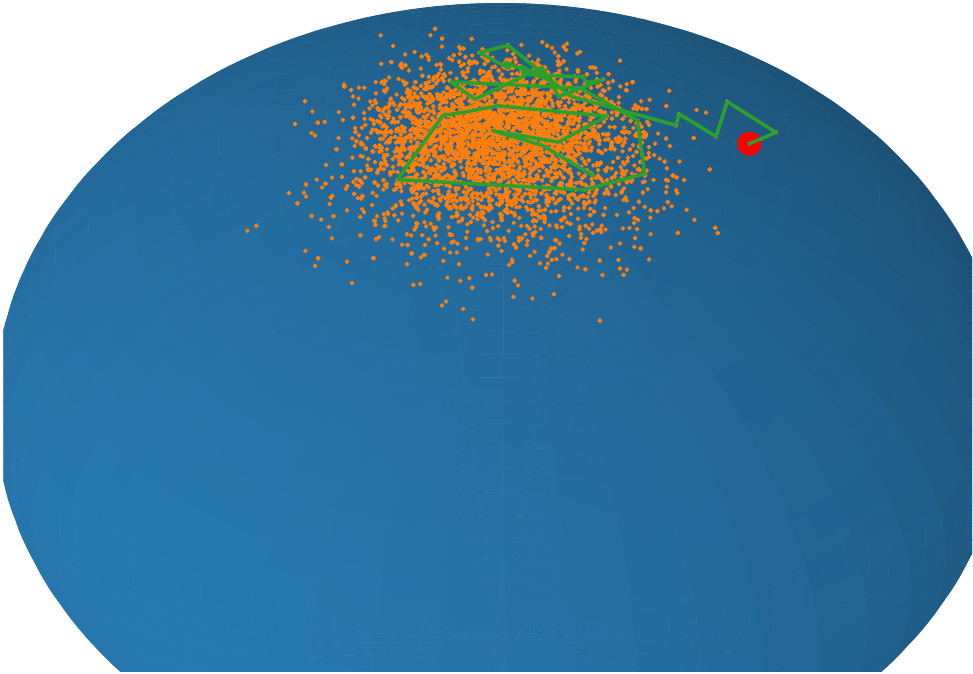}}
	\subfigure[Convergence rate]{\includegraphics[width=0.34\textwidth]{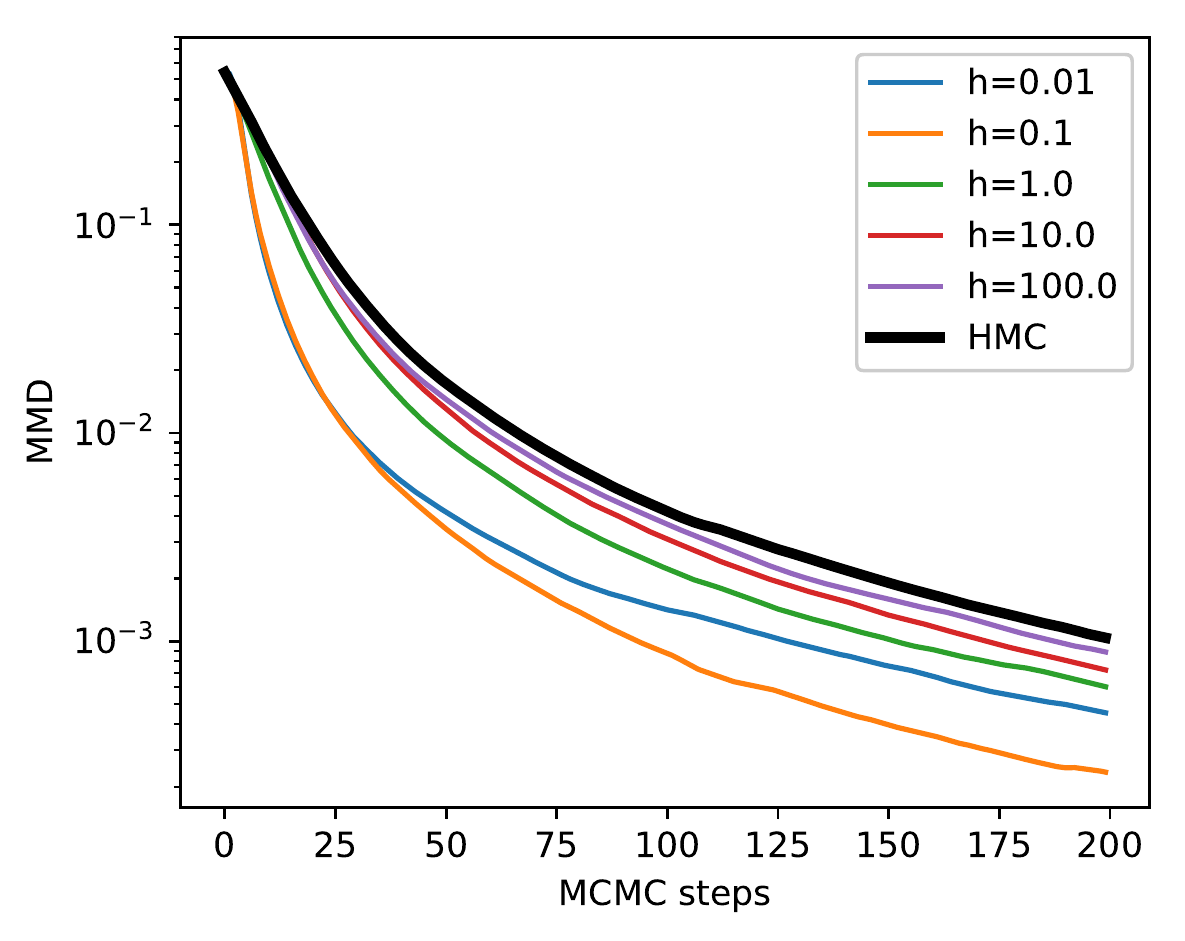}}
	\subfigure[Autocorrelations]{\includegraphics[width=0.34\textwidth]{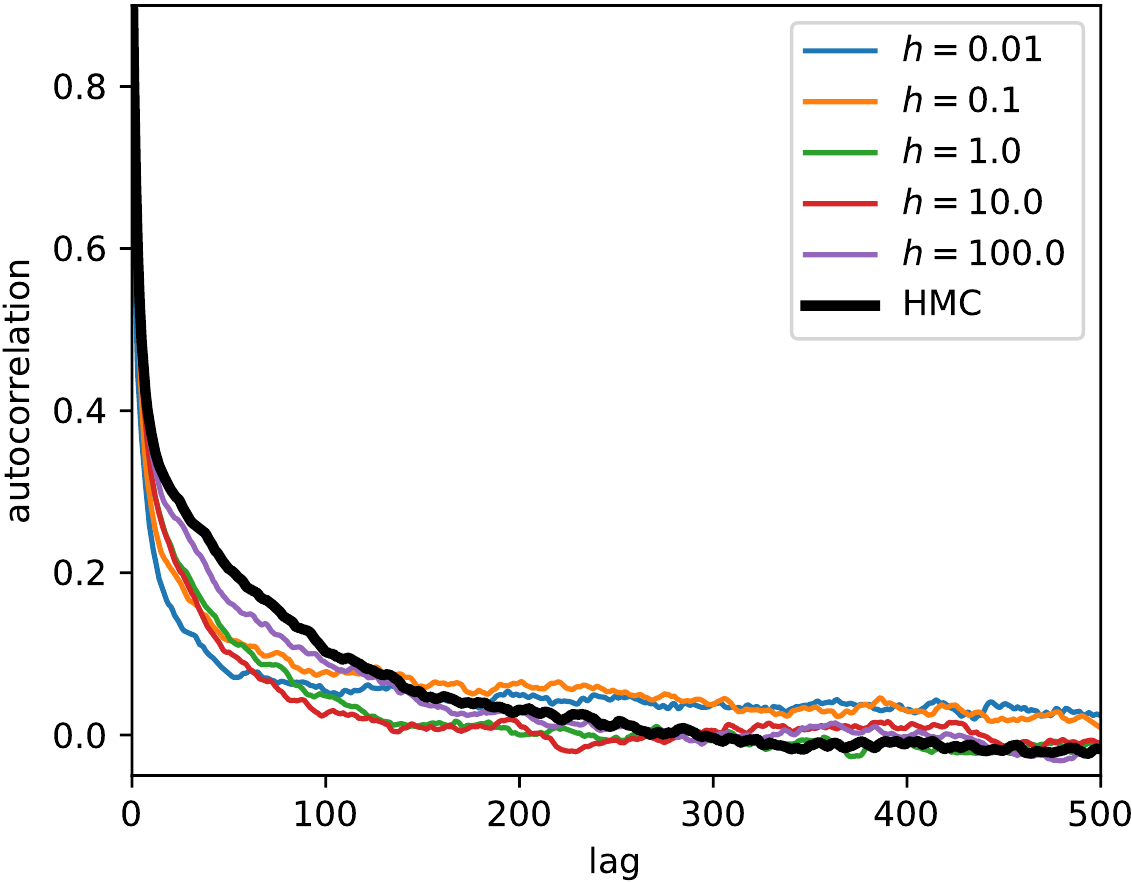}}
	\caption{This figure illustrates the irreversible MCMC algorithm on Lie group with $\G = SO(3)$.
	  For each run, we sampled $5000$ samples, shown in panel (a), along with the first one it red and the next $50$ in green. 
	  We ran several chains with several parameters $h$, corresponding to the integration time of the Langevin dynamics. 
	  For large $h$, the MCMC algorithm converges to the HMC algorithm, also displayed in black. 
	 Each line in panel $(b)$ and $(c)$ are averages over $20$ chains with the same parameters. 
     For the leapfrog integrator we use $5$ timesteps with a total time $\delta=0.5$. The MMD computation on $\G$ has been run with Gaussian kernel of variance $\sigma=1$.  } 
	\label{fig:simulation}
\end{figure}
We see that in figure \ref{fig:simulation}, small values for $h$ give MCMC algorithms with a faster convergence rate than the HMC limit, i.e., $h\to \infty$. 
This is a direct consequence from the irreversibility of the algorithm, as explained in \cite{duncan2016variance,rey2015irreversible}.
Even if a faster convergence is desirable, one has to ensure that the correct distribution is sampled, and if $h$ is taken too small, the algorithm will be close to a pure Hamiltonian dynamics, with additional reversal steps $v\to -v$ when the proposed state is rejected.
We observe this effect already for $h=0.01$ where the convergence is as fast as $h=0.1$ for the first steps of the chain, but then later slows down, as the distribution is not sampled correctly. 

{\bf Acknowledgements}\quad
{\small The authors thank A. Duncan for the useful insights that helped improve this work.
AA acknowledges EPSRC funding through award EP/N014529/1 via the EPSRC
Centre for Mathematics of Precision Healthcare. ST acknowledges the Schr\"odinger scholarship scheme for funding during this work.
AB was supported by a Roth Scholarship funded by the Department of Mathematics, Imperial College London, by EPSRC Fellowship (EP/J016934/3) and by The Alan Turing Institute under the EPSRC grant [EP/N510129/1].}

\bibliographystyle{splncs04}
\bibliography{ReferencesTHMC}

\end{document}